\documentclass{amsart}

\usepackage{graphicx, xcolor}
\usepackage{amssymb}
\newcommand{\A}{\Bbb A}
\newcommand{\BA}{\Bbb A}
\newcommand{\C}{\Bbb C}
\newcommand{\R}{\Bbb R}

\newcommand{\Z}{\Bbb Z}

\newcommand{\tr}{\mathrm{tr}\,}

\newcommand{\D}{\Delta}

\newcommand{\la}{\langle}
\newcommand{\ra}{\rangle}

\newcommand{\SL}{\mathrm{SL}_2(\C)}
\newcommand{\SLt}{\mathrm{SL}_3(\C)}
\newcommand{\SLn}{\mathrm{SL}_n(\C)}

\numberwithin{equation}{section}
\newtheorem{theorem}{Theorem}[section]
\newtheorem{proposition}[theorem]{Proposition}
\newtheorem{corollary}[theorem]{Corollary}
\newtheorem{lemma}[theorem]{Lemma}

\theoremstyle{definition}

\newtheorem{example}[theorem]{Example}
\newtheorem{remark}[theorem]{Remark}

\allowdisplaybreaks[4]

\begin{document}

\title
[Twisted Alexander polynomial and Reidemeister torsion]
{Algebraic properties of twisted Alexander polynomial and Reidemeister torsion of torus knots}

\author{Takayuki Morifuji and Anh T. Tran}
\thanks{2020 \textit{Mathematics Subject Classification}.\/ 
Primary 57K31; Secondary 57K14, 57M05}

\thanks{{\it Key words and phrases.\/ 
Twisted Alexander polynomial, Reidemeister torsion, torus knot, character variety.}}

\address{Department of Mathematics, 
Hiyoshi Campus, Keio University, 
Yokohama 223-8521, Japan}
\email{morifuji@keio.jp}

\address{Department of Mathematical Sciences, The University of Texas at Dallas, 
Richardson, TX 75080, USA}
\email{att140830@utdallas.edu}

\begin{abstract}
In this paper we prove that every coefficient of twisted Alexander polynomials of torus knots associated with irreducible $\SLn$-representations is an $\A$-valued locally constant function on the $\SLn$-character variety, where $\A$ is the ring of all algebraic integers over $\C$. Moreover, as a generalization of a recent result of Kitano and Nozaki, we show that $\SLn$-Reidemeister torsions are algebraic integers for many Seifert fibered spaces. Also, we discuss the power sums of Reidemeister torsions of torus knots for low-dimensional irreducible representations that provide a mysterious relation to TQFT. 
\end{abstract}

\maketitle
\section{Introduction}


The twisted Alexander polynomial is a generalization of the classical Alexander polynomial of knots, which was first introduced by Lin \cite{Lin01-1} for  knot groups and by Wada \cite{Wada94-1} for finitely presentable groups. Here, we note that this invariant is defined for a pair of a group and its representation. Recently, it has been widely studied and has become a powerful tool in knot theory and low dimensional topology. For recent developments on this topic and related references, see the survey papers \cite{FV10-1} and \cite{Morifuji15-1}. 

In this paper, we consider the twisted Alexander polynomial $\D_K^\rho(t)\in\C[t^{\pm1}]$ associated with an irreducible representation $\rho\colon G(K)\to \SLn$ of the group $G(K)$ of a knot $K$ in the $3$-sphere $S^3$. 
In particular, we investigate the behavior of $\D_K^\rho(t)$ as a function on the space of irreducible $\SLn$-representations. 
Regarding this property, it is known that for certain hyperbolic knots, each coefficient of $\D_K^\rho(t)$ varies continuously on the $\SL$-character variety (see \cite{GM03-1}). 
On the other hand, such a phenomenon on the $\SL$-character variety does not occur for any torus knot $K_{p,q}$, where $p$ and $q$ are coprime positive integers. 
In fact, it was shown in \cite{KtM12-1}  that every coefficient of the twisted Alexander polynomial $\D_{K_{p,q}}^\rho(t)$ is a locally constant function on the space of irreducible $\SL$-representations. 
Moreover, we verified in \cite{KMT20-1} that the same property also holds for torus links, when the twisted Alexander polynomial is associated with irreducible $\SLn$-representations that factor through $\SL$. These results are quite interesting in contrast to geometric structures (hyperbolic structure and other geometric structures) in the sense of W. Thurston. 

Now, there are two purposes of this paper. The first one is to generalize the above property of torus knots to higher-dimensional representations. In addition, we show that the coefficients of the twisted Alexander polynomial $\D_{K_{p,q}}^\rho(t)$ admit an interesting algebraic property for irreducible $\SLn$-reprensetations. Here, let us recall that a complex number $\gamma$ is called an \textit{algebraic integer} if there is a monic polynomial over $\Z$ such that $\gamma$ is a root of the polynomial. We denote the ring of all algebraic integers by $\A$, and the knot group $G(K_{p,q})$ by $G_{p,q}$. Moreover, let $\mathcal{X}_n=X(G_{p,q},\mathrm{SL}_n(\C))$ be the $\SLn$-character variety of $G_{p,q}$, and $\mathcal{X}_n^*$ its subvariety consisting of irreducible representations (see Subsection \ref{subsec:2.3} for details). Our first result of this paper is the following. 

\begin{theorem}\label{thm:main}
The coefficients of $\D_{K_{p,q}}^\rho(t)$ associated with irreducible representations $\rho\colon G_{p,q}\to\SLn$ are locally constant on 
$\mathcal{X}_n^*$. Moreover, $\D_{K_{p,q}}^\rho(t)\in\A[t^{\pm1}]$ holds.  
\end{theorem}

It might be worth mentioning that Dunfiled, Friedl and Jackson observed in \cite{DFJ12-1} that the coefficients of the twisted Alexander polynomial of a hyperbolic knot associated with a lift of the holonomy representation are often algebraic integers. 

The second purpose is to study an algebraic property of Reidemeister torsion for Seifert fibered spaces, namely, $3$-manifolds together with a decomposition as a disjoint union of circles. Reidemeister torsion is one of the classical invariants of $3$-manifolds which was introduced for the purpose of classifying lens spaces. Later, Milnor showed that the classical Alexander polynomial can be interpreted as a Reidemeister torsion. 
This framework is extended to the twisted case by using representations of the fundamental group of a $3$-manifold. Actually, the twisted Alexander polynomial is interpreted as a Reidemeister torsion at various levels (see \cite{DFJ12-1}, \cite{FV10-1}, \cite{Kitano96-1} for example). In our setting, the Reidemeister torsion $\tau_\rho(E(K))$ of the exterior $E(K)$ of a knot $K\subset S^3$ coincides with $\D_K^\rho(1)\in\C$ up to sign (see Subsection \ref{subsec:2.2} for details). 

Recently, Kitano and Nozaki studied in \cite{KN22-1} an interesting algebraic property of the $\SL$-Reidemeister torsion. In fact, they exhibited that Reidemeister torsions of many $3$-manifolds, which include Seifert fibered spaces, are not only algebraic numbers but also algebraic integers for irreducible $\SL$-representations. 
As an immediate corollary of Theorem \ref{thm:main}, we can generalize Proposition 4.9 of \cite {KN22-1} as follows. 

\begin{corollary}\label{cor:torus}
The Reidemeister torsion $\tau_\rho(E(K_{p,q}))$ associated with irreducible representations $\rho\colon G_{p,q}\to\SLn$ is locally constant on $\mathcal{X}_n^*$. Moreover, $\tau_\rho(E(K_{p,q}))$ is an algebraic integer. 
\end{corollary}

Here, we note that the exterior $E(K_{p,q})$ is a Seifert fibered space whose base orbifold has genus zero. Using the result of Kitano (see \cite[Main Theorem]{Kitano96-1}), which provides an explicit formula for the $\SLn$-Reidemeister torsion of a Seifert fibered space over the orbifold with positive genus $g$, we can show the following. 

\begin{proposition}\label{pro:g>0}
Let $N$ be a Seifert fibered space with $g\geq1$. 
For any irreducible representation $\rho\colon \pi_1(N)\to\SLn$, the Reidemeister torsion $\tau_\rho(N)$ is an algebraic integer. 
\end{proposition}

As Kitano and Nozaki pointed out in \cite{KN22-1}, there are pieces of evidence that $\SL$-Reidemeister torsions admit some interesting algebraic properties, which are quite non-trivial from their definition. Our Corollary \ref{cor:torus} and Proposition \ref{pro:g>0} suggest that this kind of algebraic properties of Reidemeister torsions hold true not only for two-dimensional representations but also for higher-dimensional representations. 

On the other hand, Yoon proved in \cite{Yoon22-1} a vanishing identity on the \textit{inverse sum} of adjoint Reidemeister torsions (with respect to a meridian) of two-bridge knots. This result implies that Reidemeister torsions of knots are imposed on certain strong algebraic restrictions. In this direction, we discuss the \textit{power sums} of Reidemeister torsions of torus knots for irreducible $\mathrm{SL}_2(\C)$-representations and their adjoint representations. We then  observe that these values of the power sums are imposed on certain algebraic restrictions. In particular, the integrality of the \textit{positive power sums} of Reidemeister torsions provides a partial answer to Conjecture 3.1 in \cite{GKY21-1} coming from a physical background. 

This paper is organized as follows. In Section~\ref{section:2}, we briefly review the definitions of the twisted Alexander polynomial $\D_K^\rho(t)$ and the Reidemeister torsion $\tau_\rho(N)$, and describe the $\SLn$-character variety of the knot group $G_{p,q}$ following \cite{MP16-1}. 
In Section~\ref{section:3}, we provide a formula for $\D_{K_{p,q}}^\rho(t)$, and prove Theorem \ref{thm:main} and Corollary \ref{cor:torus}. Moreover, we provide an explicit formula for the twisted Alexander polynomial on components of maximal dimension of the $\SLn$-character variety  of $K_{p,q}$. 
In Section~\ref{section:5}, 
we make a more detailed study of the twisted Alexander polynomial of torus knots of types $(2,q)$ and $(3,q)$ associated with irreducible $\SLt$-representations, which include the adjoint representation. 
In Section \ref{section:torsion}, we review the formula for $\SLn$-Reidemeister torsions of Seifert fibered spaces, due to Kitano \cite{Kitano96-2}, and give a proof of Proposition \ref{pro:g>0}. In Section ~\ref{section:6}, we exhibit some non-trivial algebraic restrictions on the power sums of Reidemeister torsions of torus knots. 

\section{Preriminaries}\label{section:2}

\subsection{Twisted Alexander polynomial}

Let $K$ be a knot in the $3$-sphere $S^3$. 
We denote $\pi_1(E(K))$ by $G(K)$, where $E(K)=\overline{S^3\setminus\nu(K)}$ denotes the exterior of $K$, and call it the \textit{knot group}\, of $K$. Then $G(K)$ admits a presentation of deficiency one (for example, a Wirtinger presentation), and we fix such a presentation  $G(K)=\la x_1,\ldots,x_\ell\,|\,r_1,\ldots,r_{\ell-1}\ra$. 
Let $\phi\colon G(K)\to \Z\cong\la t\ra$ be the abelianization. 

For a given irreducible representation $\rho\colon G(K)\to \SLn$, 
we extend the group homomorphism $\phi\otimes\rho\colon G(K)\to \mathrm{GL}_n(\C[t^{\pm1}])$, 
which is defined by $(\phi\otimes\rho)(x)=\phi(x)\rho(x)$ for $x\in G(K)$, to a ring homomorphism $\widetilde{\phi\otimes\rho}\colon\Z[G(K)]\to \mathrm{Mat}_n(\C[t^{\pm1}])$, 
where the target is the matrix algebra over $\C[t^{\pm1}]$. 
Let $F_\ell$ denote the free group $\la x_1,\ldots,x_\ell\ra$ 
and $\Phi\colon \Z[F_\ell]\to \mathrm{Mat}_n(\C[t^{\pm1}])$ the composition of 
the surjection $\Z[F_\ell]\to\Z[G(K)]$ induced by the presentation 
of $G(K)$ and $\widetilde{\phi\otimes\rho}$. 

Let $M$ denote the $(\ell-1)\times \ell$ matrix 
whose $(i,j)$-entry is the $n\times n$ matrix
$\Phi\left(\frac{\partial r_i}{\partial x_j}\right)
\in \mathrm{Mat}_n({\C}[t^{\pm1}])$, 
where
$\frac{\partial}{\partial x_j}$
denotes the free differential by $x_j$. 
For $1\leq k\leq \ell$,
we denote by $M_k$
the $(\ell-1)\times(\ell-1)$ matrix obtained from $M$
by removing the $k$-th column.
We regard $M_k$ as
an $n(\ell-1)\times n(\ell-1)$ matrix with coefficients in
${\C}[t^{\pm1}]$. 
Then Wada \cite{Wada94-1} defined the \textit{twisted Alexander polynomial}\, 
$\D_{K}^\rho(t)$ associated with $\rho\colon G(K)\to\SLn$ to be 
$$
\D_{K}^\rho(t)
=
\frac{\det M_k}{\det\Phi(x_k-1)},
$$
which is well-defined up to multiplication by $\pm t^{i}~(i\in\Z)$. 

\subsection{Reidemeister torsion}\label{subsec:2.2}

Let $N$ be a connected compact $3$-manifold with empty or toroidal boundary, and $\rho\colon\pi_1(N)\to\mathrm{SL}_n(\C)$ a representation. Suppose $\rho$ is acyclic; i.e., $H_*(N;\C_\rho^n)=0$, where $\C_\rho^n$ is a $\pi_1(N)$-module defined by using the action of $\rho(\pi_1(N))$ on the vector space $\C^n$. We endow $N$ with a cell decomposition so that $N$ is a CW-complex. It gives the cellular chain complex $\{C_*(N;\C_\rho^n),\partial_*\}$ with an ordered basis ${\bf c}_j$ of $C_j(N;\C_\rho^n)$ coming from $j$-cells. Next, we choose a basis ${\bf b}_j$ of $\mathrm{Im}\,\partial_{j+1}$ and its lift $\widetilde{{\bf b}}_j$ to $C_{j+1}(N;\C_\rho^n)$. The acyclicity implies that the union ${\bf b}_j\widetilde{{\bf b}}_{j-1}$ is a basis of $C_j(N;\C_\rho^n)$. We denote by $[{\bf b}_j\widetilde{\bf b}_{j-1}/{\bf c}_j]$ the change of basis matrix from ${\bf c}_j$ to ${\bf b}_j\widetilde{\bf b}_{j-1}$. The Reidemeister torsion $\tau_\rho(N)$ of $N$ associated with $\rho$ is defined by the alternating product 
$$
\tau_\rho(N)
=
\prod_{j=0}^3\big(\det [{\bf b}_j\widetilde{\bf b}_{j-1}/{\bf c}_j]\big)^{(-1)^{j+1}}\in\C\setminus\{0\},
$$
which is well-defined up to sign. When $\rho$ is non-acyclic, we simply define $\tau_\rho(N)=0$. 

If $N=E(K)$, the exterior of a knot $K$ in $S^3$, it is known that $\tau_\rho(E(K))=\D_K^\rho(1)$ holds up to sign for both the acyclic and non-acyclic cases whenever $\det\Phi(x_k-1)$ is nonzero (see the proof of \cite[Proposition 2.7]{DFJ12-1} and \cite[Proposition 3.1]{Kitano96-1}).

\subsection{$\SLn$-character variety of torus knots}\label{subsec:2.3}

We denote the torus knot of type $(p,q)$ in $S^3$ by $K_{p,q}$, where $(p,q)$ is a pair of coprime positive integers. Namely, $K_{p,q}$ is the image of the straight line $y=\frac{q}{p}x$ in $T^2=\R^2/\Z^2$ which is standardly embedded into $S^3$. 

Let us recall the knot group $G_{p,q}=G(K_{p,q})$ has a presentation $\la x,y\,|\, x^p=y^q\ra$. 
Following \cite[Sections~4 and 5]{MP16-1}, 
we quickly review a description of the character variety 
$$
\mathcal{X}_n=X(G_{p,q},\SLn),
$$ 
which is the image of the space of representations 
$\mathrm{Hom}(G_{p,q},\SLn)$ under the character map $\rho\mapsto \chi_\rho$, 
where $\chi_\rho(g)=\tr \rho(g),\,g\in G_{p,q}$. 

Let $\pi=(n_1,\overset{(a_1)}{\ldots},n_1,\ldots,n_s,\overset{(a_s)}{\ldots},n_s)$ be a \textit{partition} of $n$; that is, 
$a_1n_1+\cdots+a_sn_s=n$ with $n_1>\cdots>n_s>0$ and $a_j\geq 1$. 
Let $\Pi_n$ be the set of all partitions of $n$. We decompose the character variety 
$$
\mathcal{X}_n=\bigsqcup_{\pi\in\Pi_n}\mathcal{X}_{\pi}
$$
into locally closed subvarieties, where $\mathcal{X}_\pi$ corresponds to representations
$$
\rho=\bigoplus_{j=1}^{s}\bigoplus_{l=1}^{a_j}\rho_{j,l}\colon
G_{p,q} \to \mathrm{SL}_{n_j}(\C). 
$$
Note that the irreducible representations correspond to $\pi_0=(n)$. 
We denote $\mathcal{X}_n^*=\mathcal{X}_{\pi_0}$. 
If $(A,B)\in \mathcal{X}_n^*$, then $A$ and $B$ are both diagonalizable, and 
$A^p=B^q=\omega I_n$ for some $n$-th root of unity $\omega$, where $I_n$ is   the $n \times n$ identity matrix. Moreover, neither $A$ nor $B$ is a multiple of the identity matrix. 

Let $\epsilon_1,\ldots,\epsilon_s$ and $\varepsilon_1,\ldots,\varepsilon_l$ be the eigenvalues for $A$ and $B$ satisfying $\epsilon_i^p=\varepsilon_j^q=\omega$. 
We denote 
\begin{equation}\label{eq:0}
\kappa
=\left(
(\epsilon_1,\overset{(v_1)}{\ldots},\epsilon_1,\ldots,\epsilon_s,\overset{(v_s)}{\ldots},\epsilon_s),
(\varepsilon_1,\overset{(w_1)}{\ldots},\varepsilon_1,\ldots,\varepsilon_l,\overset{(w_l)}{\ldots},\varepsilon_l)
\right),
\end{equation}
the repeating eigenvalues according to multiplicity. 
This gives a collection of (disjoint) components 
\begin{equation}\label{eq:1}
\mathcal{X}_n^*=
\bigsqcup_\kappa\mathcal{X}_{n,\kappa}^*,
\end{equation}
each of which is irreducible and hence connected 
(see \cite[Section~5(5)]{MP16-1}). 

\section{$\mathrm{SL}_n(\C)$-Twisted Alexander polynomial of torus knots}\label{section:3}

In this section, we provide a formula for the twisted Alexander polynomial $\D_{K_{p,q}}^\rho(t)$, and prove Theorem \ref{thm:main} and Corollary \ref{cor:torus}. Moreover, we provide an explicit formula for $\D_{K_{p,q}}^\rho(t)$ on components of the maximal dimension of the $\SLn$-character variety  of $K_{p,q}$. 

\subsection{Proof of Theorem \ref{thm:main} and Corollary \ref{cor:torus}}
Let us recall the statement of our first result. 

\begin{theorem}[Theorem \ref{thm:main}]\label{thm:torusTAP}
The coefficients of $\D_{K_{p,q}}^\rho(t)$ associated with irreducible representations $\rho\colon G_{p,q}\to\SLn$ are locally constant on 
$\mathcal{X}_n^*$. Moreover, $\D_{K_{p,q}}^\rho(t)\in\A[t^{\pm1}]$ holds.  
\end{theorem}

\begin{proof}
Let $\rho\colon G_{p,q}\to \SLn$ be an irreducible representation, and $X=\rho(x)$, $Y=\rho(y)$. We may assume that $X$ is a diagonal matrix and $Y$ is conjugate to a diagonal matrix in $\SLn$. Note that $\phi(x)=t^q$ and $\phi(y)=t^p$ hold for the abelianization $\phi\colon G_{p,q}\to\Z\cong\la t\ra$.  
We set $r=x^py^{-q}$. Then $\frac{\partial r}{\partial x}=1+x + \cdots + x^{p-1}$ and
\begin{eqnarray}\label{eq:2}
\Delta_{K_{p,q}}^\rho(t) &=& \frac{\det \Phi(\frac{\partial r}{\partial x})}{\det \Phi(y-1)} \notag \\
&=& \frac{\det (I + t^q X + \cdots + t^{(p-1)q} X^{p-1})}{\det (t^pY-I)} \notag \\
&=&  \frac{\prod_{i=1}^n (1 + t^q \alpha_i + \cdots + t^{(p-1)q} \alpha_i^{p-1})}{\prod_{j=1}^n (t^p \beta_j -1)} \notag \\
&=&  \frac{\prod_{i=1}^n \frac{ t^{pq} \alpha_i^p - 1}{t^q \alpha_i - 1}}{\prod_{j=1}^n (t^p \beta_j -1)} \notag \\
&=&  \frac{(t^{pq}\omega -1)^n}{\prod_{i=1}^n (t^q \alpha_i - 1) \prod_{j=1}^n (t^p \beta_j -1)},
\end{eqnarray}
where $\omega$ is an $n$-th root of unity (i.e. $\omega^n=1$), 
$\alpha_1, \cdots, \alpha_n$ and  $\beta_1, \cdots, \beta_n$ are (repeating) eigenvalues of $X$ and $Y$ respectively. Note that $\alpha_i^p = \beta_j^q = \omega$ holds.

On each irreducible component of $\mathcal{X}_n^*$, the eigenvalues of $X$ and $Y$ (together with their multiplicities) remain unchanged, as described in (\ref{eq:1}). Thus, the coefficients of the $\mathrm{SL}_n(\C)$-twisted Alexander polynomial of the torus knot $K_{p,q}$ are constant on each irreducible component of $\mathcal{X}_n^*$. 

Next, using the notation of the eigenvalues of $X$ and $Y$ as in (\ref{eq:0}), we have
$$
\Delta_{K_{p,q}}^\rho(t) =  \frac{(t^{pq}\omega -1)^n}{ (t^q \epsilon_1 - 1)^{v_1} \cdots (t^q \epsilon_s - 1)^{v_s}  (t^p \varepsilon_1 - 1)^{w_1} \cdots (t^p \varepsilon_l - 1)^{w_l} }.
$$
This becomes a polynomial in $t$ if $v_i + w_j \le n$ for all $i, j$. Note that $v_1 + \cdots + v_s = w_1 + \cdots + w_l = n$. Since neither $X$ nor $Y$ is a multiple of the identity for any irreducible
representation $\rho$, we have $v_i < n$ and $w_j < n$. 

Assume $v_i + w_j >n$ for some $i, j$. Then the intersection of the eigenspace of $X$ corresponding to $\epsilon_i$ and the eigenspace of $Y$ corresponding to $\varepsilon_j$ gives an
invariant non-trivial subspace of $\rho$ (of dimension $v_i + w_j -n > 0$). This contradicts the irreducibility of $\rho$. 

Let $v = \max\{v_1, \ldots, v_s\}$ and $w = \max\{w_1, \ldots, w_l\}$. Then $v+w \le n$. Hence
\begin{eqnarray*}
&& \Delta_{K_{p,q}}^\rho(t) \\
&=& (t^{pq}\omega -1)^{n-v-w} \frac{(t^{pq}\omega -1)^v}{ (t^q \epsilon_1 - 1)^{v_1} \cdots (t^q \epsilon_s - 1)^{v_s} } \frac{(t^{pq}\omega -1)^w}{(t^p \varepsilon_1 - 1)^{w_1} \cdots (t^p \varepsilon_l - 1)^{w_l} }\\
&=& (t^{pq}\omega -1)^{n-v-w} (t^q \epsilon_1 - 1)^{v-v_1} \cdots (t^q \epsilon_s - 1)^{v-v_s} \left( \frac{t^{pq}\omega -1}{ (t^q \epsilon_1 - 1) \cdots (t^q \epsilon_s - 1)} \right)^v \\
&& \times \ (t^p \varepsilon_1 - 1)^{w-w_1} \cdots (t^p \varepsilon_l - 1)^{w-w_l} \left( \frac{t^{pq}\omega -1}{ (t^p \varepsilon_1 - 1) \cdots (t^p \varepsilon_l - 1)} \right)^w.
\end{eqnarray*}

Write $\omega = e^{2k\pi\sqrt{-1}/n}$, where $0 \le k \le n-1$. Since $\epsilon_i^p=\omega$ we can write $\epsilon_i = e^{2\pi\sqrt{-1}(k/n+c_i)/p}$, where $c_i \in \Z_p=\Z/p\Z$. Since $\epsilon_i$ are distinct, $c_i \in \Z_p$ are distinct. Then
$$
\frac{t^{pq}\omega -1}{ (t^q \epsilon_1 - 1) \cdots (t^q \epsilon_s - 1)} = \prod_{c \in \Z_p \setminus \{c_1, \cdots, c_s\}} (e^{2\pi\sqrt{-1}(k/n+c)/p} t^q-1) \in \C[t].
$$

Similarly, we have
$$
\frac{t^{pq}\omega -1}{ (t^p \varepsilon_1 - 1) \cdots (t^p \varepsilon_l - 1)} = \prod_{d \in \Z_q \setminus \{d_1, \cdots, d_l\}} (e^{2\pi\sqrt{-1}(k/n+d)/p} t^p-1) \in \C[t].
$$
where $d_j \in \Z_q$ are distinct.  

Since roots of unity are algebraic integers, we conclude that $\Delta_{K_{p,q}}^\rho(t) \in \A[t^{\pm 1}]$. This completes the proof of Theorem \ref{thm:torusTAP}. 
\end{proof}

\begin{remark}
When $n=2$ (and hence $\omega=\pm1$), (\ref{eq:2}) gives the same formula as in \cite[Section 4]{KtM12-1}. If we apply a similar calculation as in \cite[Section 4]{KMT20-1} to a $d$-component torus link $L=K_{d p,dq}$ in $S^3$, 
we can show that every coefficient of $\D_{L}^\rho(t_1, \cdots, t_d)$ associated with irreducible $\SLn$-representations is locally constant. 
\end{remark}

\begin{corollary}[Corollary \ref{cor:torus}]\label{cor:torusTorsion}
The Reidemeister torsion $\tau_\rho(E(K_{p,q}))$ associated with irreducible representations $\rho\colon G_{p,q}\to\SLn$ is locally constant on $\mathcal{X}_n^*$. Moreover, $\tau_\rho(E(K_{p,q}))$ is an algebraic integer. 
\end{corollary}

\begin{proof}
It is known that the evaluation of the twisted Alexander polynomial $\D_{K}^\rho(t)$ at $t=1$ coincides with the Reidemeister torsion $\tau_{\rho}(E(K))$ for any representation $\rho\colon G(K) \to \SLn$ (see \cite{Kitano96-1}). Since $\Delta_{K_{p,q}}^\rho(t) \in \A[t^{\pm 1}]$ holds by Theorem \ref{thm:torusTAP}, we have
\begin{equation}\label{eq:3}
\tau_{\rho}(E(K_{p,q}))
=  \D_{K_{p,q}}^\rho(1)
=
\frac{(\omega -1)^n}{\prod_{i=1}^n (\alpha_i - 1) \prod_{j=1}^n (\beta_j -1)}\in\A,
\end{equation}
which is locally constant on $\mathcal{X}_n^*$.
\end{proof}

\begin{remark} 
There are some remarks on Corollary \ref{cor:torusTorsion}. 
\begin{itemize}
\item[(1)] 
Corollary \ref{cor:torusTorsion} complements a result of Kitano (see \cite{Kitano96-2}), which states that $\SLn$-Reidemeister torsions of Seifert fibered spaces over an orbifold of positive genus are locally constant on the space of irreducible $\SLn$-representations (see Proposition \ref{pro:Kitano-1}). 

\item[(2)]  A representation $\rho\colon G_{p,q} \to \SLn$ is acyclic if and only if $\tau_{\rho}(E(K_{p,q})) \not= 0$. By \eqref{eq:3}, this is equivalent to $\omega \not =1$, i.e. $X^p=Y^q \not=  I_n$.

\item[(3)]
One can check that the irreducible $\SLn$-representations of $G_{p,q}$ that factor through $\SL$ always satisfy $\omega=\pm 1$ and $\omega^n=1$, see e.g. \cite[Section 2.3]{Yamaguchi13-1}.  They are acyclic if and only if $\omega \not= 1$ (by Remark 3.4(2)). This  is equivalent to $\omega=-1$ and $n$ is even. 
In this case, the formula (\ref{eq:3}) coincides with the one in \cite[Proposition 4.1]{Yamaguchi13-1}. 

\end{itemize}
\end{remark}

\subsection{$\mathrm{SL}_n(\C)$-Twisted Alexander polynomial on components of maximal dimension}\label{subsection:maximal}


The $\mathrm{SL}_n(\C)$-character variety of a torus knot has dimension at most $(n-1)^2$. If $\min\{p, q\} \ge n$, then there are $\frac{1}{n} \binom{p-1}{n-1} \binom{q-1}{n-1}$ components of maximal dimension $(n-1)^2$. They are  determined by
\begin{itemize}
\item $\alpha_1 \cdots \alpha_n = 1$, and $\alpha_i$ distinct,
\item $\beta_1 \cdots \beta_n = 1$, and $\beta_j$ distinct,
\item $\alpha_i^p=\beta_j^q=\omega$,
\item $\omega^n=1$.
\end{itemize}
Note that there are no components of dimension $(n-1)^2$ if $\min\{p, q\}<n$ (see \cite[Theorem~6.1]{MP16-1} for details).

Write $\omega = e^{2k\pi\sqrt{-1}/n}$, where $0 \le k \le n-1$. Since $\alpha_i^p=\omega$ we can write $\alpha_i = e^{2\pi\sqrt{-1}(k/n+a_i)/p}$, where $a_i \in \Z_p=\Z/p\Z$. Since $\alpha_i$ are distinct, $a_i \in \Z_p$ are distinct. The equation $\alpha_1 \cdots \alpha_n=1$ is  equivalent to $a_1+ \cdots +a_n+k \in p\Z$. 

Similarly, we can write $\beta_j = e^{2\pi\sqrt{-1}(k/n+b_j)/q}$, where $b_j \in \Z_q$ are distinct and $b_1+ \cdots +b_n+k \in q\Z$.  

Then we have the following polynomial description:
\begin{eqnarray*}
\Delta_{K_{p,q}}^\rho(t) 
&=&  \frac{(t^{pq}\omega -1)^n}{\prod_{i=1}^n (t^q \alpha_i - 1) \prod_{j=1}^n (t^p \beta_j -1)} \\
&=&   (t^{pq} \omega-1)^{n-2}   \times \frac{t^{pq}\omega -1}{\prod_{i=1}^n (t^q \alpha_i - 1)} \times \frac{t^{pq}\omega -1}{\prod_{j=1}^n (t^p \beta_j -1)} \\
&=&  (t^{pq} e^{2k\pi\sqrt{-1}/n} -1)^{n-2} \times \prod_{a \in \Z_p \setminus \{a_1, \cdots, a_n\}} (e^{2\pi\sqrt{-1}(k/n+a)/p} t^q-1)\\
&& \times  \prod_{b \in \Z_q \setminus \{b_1, \cdots, b_n\}} (e^{2\pi\sqrt{-1}(k/n+b)/q} t^p-1).
\end{eqnarray*}

\section{ $\mathrm{SL}_3(\C)$-Twisted Alexander polynomial of torus knots}\label{section:5}

In this section, we focus on components of non-maximal dimension of the $\SLt$-character variety, and provide several formulas of twisted Alexander polynomials of torus knots of types $(2,q)$ and $(3,q)$. 

The $\mathrm{SL}_3(\C)$-character variety of a torus knot has dimension at most $(3-1)^2=4$. Components are of dimension $4$ (if $\min\{p,q\} \ge 3$) or $2$ (see \cite[Proposition~8.2]{MP16-1}). There are $\frac{1}{12}(p-1)(p-2)(q-1)(q-2)$ components of dimension $4$ and the twisted Alexander polynomial on these components  is given in the previous section. More explicitly, if $\min\{p,q\} \ge 3$ then on components of dimension $4$ we have
\begin{eqnarray*}
\Delta_{K_{p,q}}^\rho(t) 
&=&  (t^{pq} e^{2k\pi\sqrt{-1}/3} -1) \prod_{a \in \Z_p \setminus \{a_1, a_2, a_3\}} (e^{2\pi\sqrt{-1}(k/3+a)/p} t^q-1)\\
&& \times  \prod_{b \in \Z_q \setminus \{b_1, b_2, b_3\}} (e^{2\pi\sqrt{-1}(k/3+b)/q} t^p-1),
\end{eqnarray*}
where $0 \le k \le 2$, $a_1, a_2, a_3 \in \Z_p$ are distinct and $a_1+ a_2 +a_3+k \in p\Z$, and $b_1, b_2, b_3 \in \Z_q$ are distinct and $b_1+ b_2 +b_3+k \in q\Z$. 

We now consider the components of dimension $2$. There are $\frac{1}{2}(p-1)(q-1)(p+q-4)$ such components (see \cite[Proposition 8.2]{MP16-1} and its proof for details). They are determined by
\begin{itemize}
\item either (1) $\alpha_1 = \alpha_2 \not= \alpha_3$ and $\beta_1,\, \beta_2, \, \beta_3$ distinct or (2) $\alpha_1,\, \alpha_2, \, \alpha_3$ distinct and $\beta_1 = \beta_2 \not= \beta_3$,
\item $\alpha_1 \alpha_2 \alpha_3=\beta_1 \beta_2 \beta_3=1$, 
\item $\alpha_i^p=\beta_j^q=\omega$,
\item $\omega^3=1$.
\end{itemize}

Assume that (1) holds. Then $\alpha_1 = \alpha_2 \not= \alpha_3$ and $\beta_1,\, \beta_2, \, \beta_3$ distinct. Write $\omega = e^{2k\pi\sqrt{-1}/3}$, where $0 \le k \le 2$. Since $\alpha_i^p=\omega$ we can write $\alpha_i = e^{2\pi\sqrt{-1}(k/3+a_i)/p}$, where $a_i \in \Z_p$. Since $\alpha_1 = \alpha_2 \not= \alpha_3$ we have $a_1=a_2 \not= a_3$.
The equation $\alpha_1 \alpha_2 \alpha_3=1$ is  equivalent to $2a_2+a_3+k \in p\Z$.  

Now we consider the $\beta_j$. Since $\beta_j^q=\omega$ we can write $\beta_j = e^{2\pi\sqrt{-1}(k/3+b_j)/q}$, where $b_j \in \Z_q$. Since $\beta_j$ are distinct, $b_j \in \Z_q$ are distinct. The equation $\beta_1 \beta_2 \beta_3=1$ is  equivalent to $b_1+b_2+b_3+k \in q\Z$. 

Then we have
\begin{eqnarray*}
\Delta_{K_{p,q}}^\rho(t) &=&  \frac{(t^{pq}\omega -1)^3}{\prod_{i=1}^3 (t^q \alpha_i - 1) \prod_{j=1}^3 (t^p \beta_j -1)} \\
&=&  (t^q \alpha_3-1) \left( \frac{t^{pq}\omega -1}{\prod_{i=2}^3 (t^q \alpha_i-1)} \right)^2 \frac{t^{pq}\omega -1}{\prod_{j=1}^3 (t^p \beta_j -1)} \\
&=& (e^{2\pi\sqrt{-1}(k/3+a_3)/p} t^q-1)  \left(\prod_{a \in \Z_p \setminus \{a_2,a_3\}} (e^{2\pi\sqrt{-1}(k/3+a)/p} t^q-1) \right)^2  \\
&& \times \, \prod_{b \in \Z_q \setminus \{b_1,b_2, b_3\}} (e^{2\pi\sqrt{-1}(k/3+b)/q} t^p-1).
\end{eqnarray*}

Case (2) can be treated in a similar manner. In this case we have $a_i \in \Z_p$ are distinct and $a_1+a_2+a_3+k \in p\Z$, and $b_1 = b_2 \not= b_3$ in $\Z_q$ and $2b_2+b_3+k \in p\Z$. Then 
\begin{eqnarray*}
\Delta_{K_{p,q}}^\rho(t) 
&=& (e^{2\pi\sqrt{-1}(k/3+b_3)/q} t^p-1)  \left(\prod_{b \in \Z_q \setminus \{b_2,b_3\}} (e^{2\pi\sqrt{-1}(k/3+b)/q} t^p-1) \right)^2  \\
&& \times \, \prod_{a \in \Z_p \setminus \{a_1,a_2, a_3\}} (e^{2\pi\sqrt{-1}(k/3+a)/p} t^q-1).
\end{eqnarray*}

\subsection{The case $p=2$} This is the $(2,q)$-torus knot. There are no components of dimension $4$ (see the first paragraph of Subsection \ref{subsection:maximal}). There are $\frac{1}{2}(q-1)(q-2)$ components of dimension $2$. 

When $p=2$, only case (1) occurs. Since  $2a_2+a_3+k \in 2\Z$, we must have $a_3=k \pmod{2}$. This implies that $e^{2\pi\sqrt{-1}(k/3+a_3)/2} = (-1)^{k}e^{k\pi\sqrt{-1}/3}$. Since $\Z_2 \setminus \{a_2,a_3\} = \emptyset$ we obtain 
\begin{eqnarray*}
\Delta_{K_{2,q}}^\rho(t) 
&=& ((-1)^{k}e^{k\pi\sqrt{-1}/3} t^q-1) \prod_{b \in \Z_q \setminus \{b_1,b_2, b_3\}} (e^{2\pi\sqrt{-1}(k/3+b)/q} t^2-1),
\end{eqnarray*}
where $0 \le k \le 2$, $b_j \in \Z_q$ are distinct and $b_1+b_2+b_3+k \in q\Z$. 

\begin{example}
$(p,q)=(2,3)$. This is the trefoil knot. There is $\frac{1}{2}(q-1)(q-2)=1$ component of irreducible representations. It is determined as follows.

Since $b_j \in \Z_3$ are distinct, we must have $\{b_1, b_2, b_3\}=\{0,1,2\}$. Then the condition $b_1+b_2+b_3+k \in 3\Z$ implies $k=0$. (In this case, $\omega=1$, $\alpha_1 =\alpha_2 =  -1$, $\alpha_3 =1$, and $\{\beta_1,\, \beta_2, \, \beta_3\} = \{1, e^{2\pi\sqrt{-1}/3}, e^{4\pi\sqrt{-1}/3}\}$.) Since $\Z_3 \setminus \{b_1,b_2, b_3\} =\emptyset$ we obtain 
\begin{eqnarray*}
\Delta_{K_{2,3}}^\rho(t) = t^3-1.
\end{eqnarray*}
\end{example}

\begin{example}
$(p,q) = (2,5)$. In this case there are $\frac{1}{2}(q-1)(q-2)=6$ components of irreducible representations. They are determined as follows: 

Since $b_1+b_2+b_3 \le 2+3+4=9$ and $0 \le k \le 2$, we have $b_1+b_2+b_3+k \le 11$. Then $b_1+b_2+b_3+k \in 5\Z$ implies that $b_1+b_2+b_3+k$ is equal to either $5$ or $10$. 

If $b_1+b_2+b_3+k = 5$ then 
\begin{itemize}
\item $k=0$ and $\{b_1, b_2, b_3\}=\{0,1,4\}$ or $\{0, 2, 3\}$.
\item $k=1$ and $\{b_1, b_2, b_3\}=\{0,1,3\}$.
\item $k=2$ and $\{b_1, b_2, b_3\}=\{0,1,2\}$.
\end{itemize}

If $b_1+b_2+b_3+k = 10$ then 
\begin{itemize}
\item $k=1$ and $\{b_1, b_2, b_3\}=\{2,3,4\}$.
\item $k=2$ and $\{b_1, b_2, b_3\}=\{1,3,4\}$.
\end{itemize}

Since \begin{eqnarray*}
\Delta_{K_{2,5}}^\rho(t) 
&=& ((-1)^{k}e^{k\pi\sqrt{-1}/3} t^5-1) \prod_{b \in \Z_5 \setminus \{b_1,b_2, b_3\}} (e^{2\pi\sqrt{-1}(k/3+b)/5} t^2-1),
\end{eqnarray*}
the twisted Alexander polynomial of these six components are listed below:

\begin{center}
        \begin{tabular}{ |c | c | c | }
        \hline
            $k$ & $\{b_1, b_2, b_3\}$ & $\Delta_{K_{2,5}}^\rho(t)$  \\ \hline 
            $0$ & $\{0,1,4\}$ & $(t^5-1) (e^{4\pi\sqrt{-1}/5} t^2-1) (e^{6\pi\sqrt{-1}/5} t^2-1)$  \\ \hline
             & $\{0,2,3\}$ & $(t^5-1) (e^{2\pi\sqrt{-1}/5} t^2-1) (e^{8\pi\sqrt{-1}/5} t^2-1)$  \\
        \hline
        $1$ & $\{0,1,3\}$ & $-(e^{\pi \sqrt{-1}/3}t^5+1) (e^{14\pi\sqrt{-1}/15} t^2-1) (e^{26\pi\sqrt{-1}/15} t^2-1)$  \\ \hline
          & $\{2,3,4\}$ & $-(e^{\pi \sqrt{-1}/3}t^5+1) (e^{2\pi\sqrt{-1}/15} t^2-1) (e^{8\pi\sqrt{-1}/15} t^2-1)$  \\ \hline
        $2$ & $\{0,1,2\}$ & $(e^{2\pi \sqrt{-1}/3}t^5-1) (e^{22\pi\sqrt{-1}/15} t^2-1) (e^{28\pi\sqrt{-1}/15} t^2-1)$  \\ \hline  
         & $\{1,3,4\}$ & $(e^{2\pi \sqrt{-1}/3}t^5-1) (e^{4\pi\sqrt{-1}/15} t^2-1) (e^{16\pi\sqrt{-1}/15} t^2-1)$  \\ \hline
        \end{tabular}
    \end{center}
\end{example}

\subsection{The case $p=3$} There are $\frac{1}{6}(q-1)(q-2)$ components of dimension $4$. On these components, since $a_1, a_2, a_3 \in \Z_3$ are distinct we have $\{a_1, a_2, a_3\} = \{0,1,2\}$. Then $a_1+ a_2 +a_3+k \in 3\Z$ implies that $k=0$. Hence 
\begin{eqnarray*}
\Delta_{K_{3,q}}^\rho(t) 
&=&  (t^{3q}  -1)  \prod_{b \in \Z_q \setminus \{b_1, b_2, b_3\}} (e^{2\pi\sqrt{-1}b/q} t^3-1),
\end{eqnarray*}
where $b_1, b_2, b_3 \in \Z_q$ are distinct and $b_1+ b_2 +b_3 \in q\Z$. 

There are $(q-1)^2$  components of dimension $2$. On these components, either (1) $\alpha_1 = \alpha_2 \not= \alpha_3 \in \Z_3$ and $\beta_1,\, \beta_2, \, \beta_3 \in \Z_q$ distinct, or (2) $\alpha_1,\, \alpha_2, \, \alpha_3 \in \Z_3$ are distinct and $\beta_1 = \beta_2 \not= \beta_3 \in \Z_q$.

Case (1): The condition $2a_2+a_3+k \in 3\Z$ is equivalent to $a_2 = a_3 +k \pmod{3}$. Then $a_2 \not= a_3$ becomes $k \not= 0$. Hence we have the following possibilities: 
\begin{itemize}
\item $k=1$ and $(a_2, a_3)=(1,0)$ or $(2,1)$ or $(0,2)$.
\item $k=2$ and $(a_2, a_3)=(2,0)$ or $(0,1)$ or $(1,2)$.
\end{itemize}

Let 
$$
P = \prod_{b \in \Z_q \setminus \{b_1,b_2, b_3\}} (e^{2\pi\sqrt{-1}(k/3+b)/q} t^3-1),
$$ 
where $b_j \in \Z_q$ are distinct and $b_1+b_2+b_3+k \in q\Z$. Since 
$$
\Delta_{K_{3,q}}^\rho(t) 
= (e^{2\pi\sqrt{-1}(k/3+a_3)/3} t^q-1)  \left(\prod_{a \in \Z_3 \setminus \{a_2,a_3\}} (e^{2\pi\sqrt{-1}(k/3+a)/3} t^q-1) \right)^2  P,
$$
the corresponding twisted Alexander polynomial are listed as follows:
\begin{center}
        \begin{tabular}{ |c | c | c | }
        \hline
            $k$ & $(a_2, a_3)$ & $\Delta_{K_{3,q}}^\rho(t)$  \\ \hline 
            $1$ & $(1,0)$ & $(e^{2\pi\sqrt{-1}/9} t^q-1) (e^{14\pi\sqrt{-1}/9} t^q-1)^2  P$ \\ \hline
             & $(2,1)$ & $(e^{8\pi\sqrt{-1}/9} t^q-1)(e^{2\pi\sqrt{-1}/9} t^q-1)^2 P$\\
        \hline
        & $(0,2)$ & $(e^{14\pi\sqrt{-1}/9} t^q-1)(e^{8\pi\sqrt{-1}/9} t^q-1)^2 P$ \\
        \hline
        $2$ & $(2,0)$ & $(e^{4\pi\sqrt{-1}/9} t^q-1) (e^{10\pi\sqrt{-1}/9} t^q-1)^2  P$ \\ \hline
          & $(0,1)$ & $(e^{10\pi\sqrt{-1}/9} t^q-1) (e^{16\pi\sqrt{-1}/9} t^q-1)^2  P$ \\ \hline
          & $(1,2)$ & $(e^{16\pi\sqrt{-1}/9} t^q-1) (e^{4\pi\sqrt{-1}/9} t^q-1)^2  P$ \\ \hline
        \end{tabular}
    \end{center}

Case (2): Since $a_1, a_2, a_3 \in \Z_3$ are distinct, we have $\{a_1, a_2, a_3\}= \{0,1,2\}$. Then $a_1+a_2+a_3+k \in 3\Z$ implies that $k=0$. 

We now consider $b_j$. We have $b_1 = b_2 \not= b_3$ in $\Z_q$. Since $k=0$, the condition $2b_2+b_3+k \in q\Z$ becomes $2b_2+b_3 \in q\Z$. 

Since 
\begin{eqnarray*}
\Delta_{K_{3,q}}^\rho(t) 
&=& (e^{2\pi\sqrt{-1}(k/3+b_3)/q} t^3-1)  \left(\prod_{b \in \Z_q \setminus \{b_2,b_3\}} (e^{2\pi\sqrt{-1}(k/3+b)/q} t^3-1) \right)^2  \\
&& \times \, \prod_{a \in \Z_3 \setminus \{a_1,a_2, a_3\}} (e^{2\pi\sqrt{-1}(k/3+a)/p} t^q-1),
\end{eqnarray*}
we have 
\begin{eqnarray*}
\Delta_{K_{3,q}}^\rho(t) 
&=& (e^{2\pi\sqrt{-1}b_3/q} t^3-1) \left(\prod_{b \in \Z_q \setminus \{b_2,b_3\}} (e^{2\pi\sqrt{-1}b/q} t^3-1) \right)^2.
\end{eqnarray*}

\begin{example}
$(p,q) = (3,4)$. There is $\frac{1}{6}(q-1)(q-2)=1$ component of dimension $4$. It is determined as follows. Since $b_1 + b_2 +  b_3 \le 1+2+3=6$ and $b_1+ b_2 +b_3 \in 4\Z$, we have $b_1+ b_2 +b_3=4$. This implies that $\{b_1, b_2, b_3\}=\{0, 1, 3\}$.  Hence
\begin{eqnarray*}
\Delta_{K_{3,4}}^\rho(t) 
&=&  (t^{3q}  -1)  \prod_{b \in \Z_q \setminus \{b_1, b_2, b_3\}} (e^{2\pi\sqrt{-1}b/q} t^3-1) \\
&=& -(t^{12}  -1)  ( t^3+1).
\end{eqnarray*}

There are $(q-1)^2=9$  components of dimension $2$. For case (1) we have $k \in \{1,2\}$, $b_j \in \Z_4$ are distinct and $b_1+b_2+b_3+k \in 4\Z$. Since $b_1+b_2+b_3+k \le 1 + 2 + 3 +2 =8$, we have $b_1+b_2+b_3+k=4$ or $8$. 

If $b_1+b_2+b_3+k=4$ then $\{b_1, b_2, b_3\}=\{0, 1, 2\}$ and $k=1$. Then 
$$
P = \prod_{b \in \Z_4 \setminus \{b_1,b_2, b_3\}} (e^{2\pi\sqrt{-1}(k/3+b)/4} t^3-1) = e^{5\pi\sqrt{-1}/3} t^3-1. 
$$

If $b_1+b_2+b_3+k=8$ then $\{b_1, b_2, b_3\}=\{1,2,3\}$ and $k=2$. Then 
$$
P = \prod_{b \in \Z_4 \setminus \{b_1,b_2, b_3\}} (e^{2\pi\sqrt{-1}(k/3+b)/4} t^3-1) = e^{\pi\sqrt{-1}/3} t^3-1. 
$$

For the six components of dimension $2$ in case (1) we have
\begin{center}
        \begin{tabular}{ |c | c | c | }
        \hline
            $k$ & $(a_2, a_3)$ & $\Delta_{K_{3,4}}^\rho(t)$  \\ \hline 
            $1$ & $(1,0)$ & $(e^{2\pi\sqrt{-1}/9} t^4-1) (e^{14\pi\sqrt{-1}/9} t^4-1)^2  (e^{5\pi\sqrt{-1}/3} t^3-1) $ \\ \hline
             & $(2,1)$ & $(e^{8\pi\sqrt{-1}/9} t^4-1)(e^{2\pi\sqrt{-1}/9} t^4-1)^2 (e^{5\pi\sqrt{-1}/3} t^3-1)$\\
        \hline
        & $(0,2)$ & $(e^{14\pi\sqrt{-1}/9} t^4-1)(e^{8\pi\sqrt{-1}/9} t^4-1)^2 (e^{5\pi\sqrt{-1}/3} t^3-1)$ \\
        \hline
        $2$ & $(2,0)$ & $(e^{4\pi\sqrt{-1}/9} t^4-1) (e^{10\pi\sqrt{-1}/9} t^4-1)^2  (e^{\pi\sqrt{-1}/3} t^3-1)$ \\ \hline
          & $(0,1)$ & $(e^{10\pi\sqrt{-1}/9} t^4-1) (e^{16\pi\sqrt{-1}/9} t^4-1)^2  (e^{\pi\sqrt{-1}/3} t^3-1)$ \\ \hline
          & $(1,2)$ & $(e^{16\pi\sqrt{-1}/9} t^4-1) (e^{4\pi\sqrt{-1}/9} t^4-1)^2  (e^{\pi\sqrt{-1}/3} t^3-1)$ \\ \hline
        \end{tabular}
    \end{center}

For case (2) we have $b_2 \not= b_3$ in $\Z_4$ and $2b_2+b_3 \in 4\Z$. Since $2b_2+b_3 \le 2(3)+2 =8$, we have $2b_2+b_3=4$ or $8$. 

If $2b_2+b_3=4$ then $(b_2, b_3)=(1, 2)$ or $(2,0)$. 

If $2b_2+b_3=8$ then $(b_2, b_3)=(3, 2)$. 

Since 
\begin{eqnarray*}
\Delta_{K_{3,4}}^\rho(t) 
&=& (e^{2\pi\sqrt{-1}b_3/4} t^3-1) \left(\prod_{b \in \Z_4 \setminus \{b_2,b_3\}} (e^{2\pi\sqrt{-1}b/4} t^3-1) \right)^2
\end{eqnarray*} 
we have 
\begin{center}
\begin{tabular}{ | c | c | }
\hline
            $(b_2, b_3)$ & $\Delta_{K_{3,4}}^\rho(t)$  \\ \hline 
            $(1,2)$ & $-(t^3+1)(t^3-1)^2 (\sqrt{-1}\, t^3+1)^2$ \\ \hline
            $(2,0)$ & $(t^3-1)(\sqrt{-1} \, t^3-1)^2 (\sqrt{-1} \, t^3+1)^2 = (t^3-1)(t^6+1)^2$\\
        \hline
            $(3,2)$ & $-(t^3+1)(t^3-1)^2 (\sqrt{-1} \, t^3-1)^2$ \\ \hline
        \end{tabular}
    \end{center}
\end{example}

\section{Seifert fibered spaces with $g \ge 1$}\label{section:torsion}

In this section, we provide a proof of Proposition \ref{pro:g>0} using the result of Kitano (see \cite{Kitano96-2}). We first recall the notation of \textit{Seifert index} according to \cite{Orlik72-1}. 

Let $N$ denote the orientable Seifert fibered space given by the following Seifert index $\{v,(o,g);(\alpha_1,\beta_1),\ldots,(\alpha_m,\beta_m)\}$. We assume that the genus $g$ of the base orbifold is positive. Using the Seifert index, the fundamental group $\pi_1(N)$ is given by the following presentation: 
\begin{align*}
\pi_1(N)
=
\la
a_1,b_1,\ldots,a_g,b_g,q_1,\ldots,q_m,h\,|\,&
[a_i,h]=[b_i,h]=[q_i,h]=1,q_j^{\alpha_j}h^{\beta_j}=1,\\
&q_1\cdots q_m[a_1,b_1]\cdots[a_g,b_g]=h^v
\ra.
\end{align*}

Let $\rho\colon \pi_1(N)\to \SLn$ be an irreducible representation. By the irreducibility of $\rho$, there exists $n$-th root  of unity $\omega$ such that $\rho(h)=\omega I_n$. 
Let $e_{j,1},\ldots,e_{j,n}$ denote the eigenvalues of $\rho(q_j)$. Then, Kitano derived the following formula. Here note that we take the inverse of Reidemeister torsion appeared in \cite{Kitano96-2}. 

\begin{proposition}[{\cite[Main Theorem,~Corollary A]{Kitano96-2}}]\label{pro:Kitano-1}
The Reidemeister torsion $\tau_\rho(N)$ is given by
$$
\tau_\rho(N)
=
\frac{(\omega-1)^{n(m+2g-2)}}
{\prod_{j=1}^{m}\left(\omega^{\nu_j}e_{j,1}^{\mu_j}-1\right)\cdots
\left(\omega^{\nu_j}e_{j,n}^{\mu_j}-1\right)}
$$
where $\mu_j,\nu_j\in\Z$ such that $\alpha_j\nu_j-\beta_j\mu_j=-1$ and $0<\mu_j<\alpha_j$. Moreover, $\tau_\rho(N)$ is a constant function on each connected component of the space of irreducible $\SLn$-representations. 
\end{proposition}

Let us recall the statement of Proposition \ref{pro:g>0}. 

\begin{proposition}[Proposition \ref{pro:g>0}]\label{pro:main-2}
Let $N$ be a Seifert fibered space with $g\geq1$. 
For any irreducible representation $\rho\colon \pi_1(N)\to\SLn$, the Reidemeister torsion $\tau_\rho(N)$ is an algebraic integer. 
\end{proposition}

\begin{proof}
By Proposition \ref{pro:Kitano-1}, we have
\begin{eqnarray*}
\tau_{\rho}(N) &=& \frac{(\omega-1)^{n(m+2g-2)}} {\prod_{j=1}^m (\omega^{\nu_j} e_{j,1}^{\mu_j} -1) \cdots (\omega^{\nu_j} e_{j,n}^{\mu_j} -1)} \\
&=& (\omega-1)^{n(2g-2)} \prod_{j=1}^m \frac{(\omega-1)^{n} }{ (\omega^{\nu_j} e_{j,1}^{\mu_j} -1) \cdots (\omega^{\nu_j} e_{j,n}^{\mu_j} -1)}. 
\end{eqnarray*}
Since $e_{j,k}^{\alpha_j} \omega^{\beta_j}=1$ and $\alpha_j \nu_j - \beta_j \mu_j =-1$, we have $\omega^{-1} = \omega^{\alpha_j \nu_j} \omega^{- \beta_j \mu_j} =\omega^{\alpha_j \nu_j} e_{j,k}^{\alpha_j \mu_j} = (\omega^{\nu_j} e_{j,k}^{\mu_j} )^{\alpha_j }$. This implies that 
$$
\frac{\omega^{-1}-1}{ \omega^{\nu_j} e_{j,k}^{\mu_j} -1 } = \frac{(\omega^{\nu_j} e_{j,k}^{\mu_j} )^{\alpha_j } - 1}{ \omega^{\nu_j} e_{j,k}^{\mu_j} -1 } \in \BA.
$$
Hence $\tau_{\rho}(N) \in \BA$ holds. This completes the proof of Proposition \ref{pro:main-2}. 
\end{proof}

\section{Power sum of Reidemeister torsions}\label{section:6}

In this section, we discuss the power sums of Reidemeister torsions of torus knots for irreducible representations $\rho\colon G_{p,q}\to\SL$ and their adjoint representations $\mathrm{ad} \circ \rho \colon G_{p,q}\to\SLt$. 

Recall that $G_{p,q}=\la x,y\,|\, x^p=y^q\ra$. Let $\mu=x^{-r}y^{s}\in G_{p,q}$ be a meridian of $K_{p,q}$, where $(r,s)$ is any pair of integers satisfying $ps-qr=1$. 

When $n=2$, the $\SL$-character variety $\mathcal{X}^*_2$ consists of $\frac{(p-1)(q-1)}{2}$ components, which are indexed by the following pairs of integers $(a, b)$ and denoted by $\mathcal{X}^*_{2, (a,b)}$:
\begin{itemize}
\item $0<a < p$, $0< b < q$.
\item $a \equiv b \pmod{2}$.
\item For any $[\rho] \in \mathcal{X}^*_{2, (a,b)}$ we have $\tr \rho(x) =2 \cos \frac{\pi a}{p}$ and $\tr \rho(y) = 2 \cos \frac{\pi b}{q}$. Moreover, the irreducible $\SL$-representation $\rho$ sends $x^p=y^q$ to $(-I)^a$. 
\item $\tr \rho(\mu) \not= 2 \cos \left( \frac{\pi r a}{p} \pm \frac{\pi sb}{q} \right)$.
\end{itemize}
Note that each component $\mathcal{X}^*_{2, (a,b)}$ of $\mathcal{X}^*_2$ is a complex line $\C$ with local parameter $\tr \rho(\mu)$. See \cite{Joh}. 

\subsection{$\mathrm{SL}_2(\C)$-Reidemeister torsions}
We first prepare  the following lemma.

\begin{lemma} \label{id}
For $p \ge 2$ and $m \ge 0$ we have
$$
\sum_{\substack{0< a < 2p \\ a \emph{ odd}}}(4\sin^2 \frac{a \pi}{2p})^m = p  \sum_{l=-\lfloor \frac{m}{p} \rfloor}^{\lfloor \frac{m}{p} \rfloor} (-1)^{(p+1)l}\binom{2m}{m + pl}.
$$
\end{lemma}

\begin{proof}
Let $I = \sum_{\substack{0< a < 2p \\ a \text{ odd}}} (2\sin \frac{a \pi}{2p})^{2m}$. Since $2\sin \frac{a \pi}{2p} = \frac{ e^{i \frac{a \pi}{2p}} - e^{-i \frac{a \pi}{2p}}  }{i}$ we have 
\begin{eqnarray*}
I &=& (-1)^m \sum_{\substack{0< a < 2p \\ a \text{ odd}}} \sum_{k=0}^{2m} (-1)^k \binom{2m}{k} e^{i \frac{(2m-k) a \pi}{2p}} e^{-i \frac{ka \pi}{2p}} \\
&=& (-1)^m \sum_{k=0}^{2m} (-1)^k \binom{2m}{k} \sum_{\substack{0< a < 2p \\ a \text{ odd}}}  e^{i \frac{(m-k) a \pi}{p}}.
\end{eqnarray*}
Note that if $p \nmid (k-m)$ then 
$$
\sum_{\substack{0< a < 2p \\ a \text{ odd}}}  e^{i \frac{(m-k) a \pi}{p}} = 
\frac{ e^{i \frac{(m-k) (2p+1)\pi}{p}} - e^{i \frac{(m-k) \pi}{p}}  }{ e^{i \frac{(m-k) 2 \pi}{p}}  - 1 } =0.
$$ 
Moreover, if $p \mid (k-m)$ then $\sum_{\substack{0< a < 2p \\ a \text{ odd}}}  e^{i \frac{(m-k) a \pi}{p}} = (-1)^{\frac{k-m}{p}} p$. Hence
\begin{eqnarray*}
I &=& \sum_{\substack{0 \le k \le 2m \\ p \mid (k-m)}} (-1)^{k-m} \binom{2m}{k} (-1)^{\frac{k-m}{p}} p \\
&=& p  \sum_{l=-\lfloor \frac{m}{p} \rfloor}^{\lfloor \frac{m}{p} \rfloor} (-1)^{(p+1)l}\binom{2m}{m + pl}. 
\end{eqnarray*}
This completes the proof of Lemma \ref{id}. 
\end{proof}

\subsubsection{Negative power sum}
An irreducible $\SL$-representation $\rho$ on the component $\mathcal{X}^*_{2, (a,b)}$ is acyclic if and only if $a$ and $b$ are  odd. In which case, its Reidemeister torsion  is given by 
$$
\tau_\rho :=
\tau_\rho(E(K_{p,q})) = \frac{1} { 4  \sin^2 \frac{a\pi}{2p} \sin^2 \frac{b\pi}{2q}}.
$$ 
See \cite{Joh, Yamaguchi13-1}. 

For a fixed complex number $c\in \C$ (different from $2 \cos \pi (r a/p \pm sb/q)$ for all pairs $(a,b)$), there is exactly one conjugacy class $[\rho] \in \mathcal{X}^*_{2, (a,b)}$ such that $\tr \rho(\mu) = c$ for each pair $(a,b)$. Then the negative power sum of Reidemeister torsions of acyclic $[\rho] \in \mathcal{X}^*_2$ with fixed $\tr \rho(\mu)=c$ is equal to 
$$
\sum_{\substack {\text{acyclic }[\rho] \in \mathcal{X}^*_2 \\ \tr \rho(\mu)=c}} \frac{1}{(\tau_\rho)^{m}} = 
\left( \sum_{\substack{0< a < p \\ a \text{ odd}}}  (2\sin^2 \frac{a\pi}{2p})^m \right) \left( \sum_{\substack{0< b < q \\ b \text{ odd}}}  (2\sin^2 \frac{b\pi}{2q})^m \right)
$$
where $m$ is a non-negative integer.

Let $\varepsilon_p = p \pmod{2} \in \{0,1\}$. Since $\sin \frac{(2p-a)\pi}{2p} = \sin \frac{a\pi}{2p}$, by Lemma \ref{id} we have 
\begin{eqnarray*}
2 \sum_{\substack{0< a < p \\ a \text{ odd}}}  (2\sin^2 \frac{a\pi}{2p})^m &=& \sum_{\substack{0< a < 2p \\ a \text{ odd}}}  (2\sin^2 \frac{a\pi}{2p})^m - 2^m \varepsilon_p  \\
&=& \frac{p}{2^m}  \sum_{l=-\lfloor \frac{m}{p} \rfloor}^{\lfloor \frac{m}{p} \rfloor}  (-1)^{(p+1)l} \binom{2m}{m + pl} - 2^m \varepsilon_p.
\end{eqnarray*}
This implies that 
$$
 \sum_{\substack{0< a < p \\ a \text{ odd}}}  (2\sin^2 \frac{a\pi}{2p})^m =  \frac{p}{2^{m+1}}  \sum_{l=-\lfloor \frac{m}{p} \rfloor}^{\lfloor \frac{m}{p} \rfloor} (-1)^{(p+1)l} \binom{2m}{m + pl} - 2^{m-1} \varepsilon_p \in \frac{1}{2^m}\Z,
$$
because $(-1)^{(p+1)l}\begin{pmatrix}2m\\m+pl\end{pmatrix}=
(-1)^{-(p+1)l}\begin{pmatrix}2m\\m-pl\end{pmatrix}$ and 
$\begin{pmatrix}2m\\m\end{pmatrix} = 2 \begin{pmatrix}2m-1\\m-1\end{pmatrix} \in2\Z$ hold. 
Hence 
\begin{eqnarray*}
\sum_{\substack {\text{acyclic } [\rho] \in \mathcal{X}^*_2 \\ \tr \rho(\mu)=c}}  \frac{1}{(\tau_\rho)^{m}} &=& \left( \frac{p}{2^{m+1}}  \sum_{l=-\lfloor \frac{m}{p} \rfloor}^{\lfloor \frac{m}{p} \rfloor} (-1)^{(p+1)l} \binom{2m}{m + pl} - 2^{m-1} \varepsilon_p \right) \\
&& \times \left( \frac{q}{2^{m+1}}  \sum_{l=-\lfloor \frac{m}{q} \rfloor}^{\lfloor \frac{m}{q} \rfloor}  (-1)^{(q+1)l} \binom{2m}{m + ql} - 2^{m-1} \varepsilon_q \right)
\in\frac{1}{4^m}\Z.
\end{eqnarray*}
In particular, the inverse sum ($m=1$) of Reidemeister torsions of acyclic $[\rho] \in \mathcal{X}^*_2$ with fixed $\tr \rho(\mu)$  is equal to $(\frac{p}{2} -\varepsilon_p) (\frac{q}{2} -\varepsilon_q)$. As a special case, when $q=2$ (and $p$ is odd), $K_{p,2}$ is the two-bridge knot corresponding to the fraction $p/1$ and by Remark 2.18 of \cite{MY25-1}, the inverse sum of Reidemeister torsions of parabolic representations ($\tr \rho(\mu)=2$) is equal to $\frac{p-2}{2}$. 
Moreover, since $\sum_{\substack{0< b < q \\ b \text{ odd}}}  (2\sin^2 \frac{b\pi}{2q})^m = 1$
holds for $q=2$, we obtain 
\begin{eqnarray*}
\sum_{\substack {\text{acyclic } [\rho] \in \mathcal{X}^*_2 \\ \tr \rho(\mu)=c}}  \frac{1}{(\tau_\rho(E(K_{p,2}))^{m}} 
&=& \sum_{\substack{0< a < p \\ a \text{ odd}}}  (2\sin^2 \frac{a\pi}{2p})^m \\
&=&
\frac{p}{2^{m+1}}  \sum_{l=-\lfloor \frac{m}{p} \rfloor}^{\lfloor \frac{m}{p} \rfloor} (-1)^{(p+1)l} \binom{2m}{m + pl} - 2^{m-1} \in\frac{1}{2^m}\Z,
\end{eqnarray*}
because $\varepsilon_p=1$. 

\subsubsection{Positive power sums} We will express positive power sums of Reidemeister torsions in terms of ranks of certain modules appearing in TQFT (see \cite{BHMV95-1}). We have
$$
\sum_{\substack {\text{acyclic }[\rho] \in \mathcal{X}^*_2 \\ \tr \rho(\mu)=c}}  (\tau_\rho)^{m} = 
 \sum_{\substack{0< a < p \\ a \text{ odd}}}  \left( \frac{1}{2\sin^2 \frac{a\pi}{2p}} \right)^m  \sum_{\substack{0< b < q \\ b \text{ odd}}}  \left( \frac{1}{2\sin^2 \frac{b\pi}{2q}} \right)^m.
$$
Corollary 1.16 in \cite{BHMV95-1} says that for $m \ge -1$ we have \textit{Verlinde's formula}  
$$
\text{rank } V_p(\Sigma_{m+1}) = \ \sum_{\substack{0< a < p \\ a \text{ even}}} \left( \frac{p}{4\sin^2 \frac{a\pi}{p}} \right)^m,
$$
where $V_p(\Sigma_{m+1}),\,p\geq3$, denotes the module over a commutative ring with unit and conjugation associated to a closed surface of genus $m+1$ equipped with the empty link. These modules appear in the framework of TQFT derived from the Kauffman bracket (see \cite{BHMV95-1} for the precise definition and its properties). 
In particular, $
\text{rank } V_{2p}(\Sigma_{m+1}) = \sum_{0 < a < p } \left( \frac{p}{2\sin^2 \frac{a\pi}{p}} \right)^m$. Then
\begin{eqnarray*}
 \sum_{\substack{0< a < p \\ a \text{ odd}}}  \left( \frac{p}{4\sin^2 \frac{a\pi}{p}} \right)^m &=& \sum_{0 < a < p } \left( \frac{p}{4\sin^2 \frac{a\pi}{p}} \right)^m -  \sum_{\substack{0< a < p \\ a \text{ even}}} \left( \frac{p}{4\sin^2 \frac{a\pi}{p}} \right)^m \\
&=& 2^{-m} \text{rank } V_{2p}(\Sigma_{m+1}) - \text{rank } V_p(\Sigma_{m+1}).
\end{eqnarray*}
This implies that 
\begin{eqnarray*}
2  \sum_{\substack{0< a < p \\ a \text{ odd}}} \left( \frac{1}{2\sin^2 \frac{a\pi}{2p}} \right)^m &=&  \sum_{\substack{0< a < 2p \\ a \text{ odd}}}  \left( \frac{1}{2\sin^2 \frac{a\pi}{2p}} \right)^m - 2^{-m} \varepsilon_p  \\
&=& (2p)^{-m} \text{rank } V_{4p}(\Sigma_{m+1}) - p^{-m}  \text{rank } V_{2p}(\Sigma_{m+1}) - 2^{-m} \varepsilon_p.
\end{eqnarray*}
Hence, for $m \ge -1$ we obtain
\begin{eqnarray*}
&& \sum_{\substack {\text{acyclic }[\rho] \in \mathcal{X}^*_2 \\ \tr \rho(\mu)=c}}    (\tau_\rho)^{m} \\
&=& \frac{1}{4} \left( (2p)^{-m} \text{rank } V_{4p}(\Sigma_{m+1}) - p^{-m}  \text{rank } V_{2p}(\Sigma_{m+1}) - 2^{-m} \varepsilon_p\right) \\
&& \times \left( (2q)^{-m} \text{rank } V_{4q}(\Sigma_{m+1}) - q^{-m}  \text{rank } V_{2q}(\Sigma_{m+1}) - 2^{-m} \varepsilon_q \right)
\in \frac{1}{4(4pq)^m}\Z.
\end{eqnarray*}

Note that $\text{rank } V_{p}(\Sigma_{0})=1$. If $m=-1$, we again conclude that the inverse sum of torsions is equal to $(\frac{p}{2} -\varepsilon_p) (\frac{q}{2} -\varepsilon_q)$.

\subsection{Adjoint Reidemeister torsions}
We prepare  the following lemma. 

\begin{lemma} \label{id1}
For $p \ge 2$ and $m \ge 0$ we have
$$
\sum_{0< a < p} (4\sin^2 \frac{a \pi}{p})^m = p  \sum_{l=-\lfloor \frac{m}{p} \rfloor}^{\lfloor \frac{m}{p} \rfloor}  (-1)^{pl} \binom{2m}{m +pl}.
$$
\end{lemma}

\begin{proof}
The proof is similar to that of Lemma \ref{id}.
\end{proof}

Let us recall that the adjoint Reidemeister torsion of the torus knot $K=K_{p,q}$ with respect to the meridian $\mu$ is given by 
$$
\tau_{\mathrm{ad}\circ\rho} := \tau_{\mathrm{ad} \circ \rho}(K,\mu)= \pm \frac{pq} {16\sin^2 \frac{a\pi}{p} \sin^2 \frac{b\pi}{q}}$$ 
where $0<a<p$, $0<b<q$, $a \equiv b \pmod{2}$ and $[\rho] \in \mathcal{X}^*_{2, (a,b)}$. See \cite[Section 6]{Dub} and \cite[Section 7]{TY21-1}. 

\subsubsection{Negative power sum}
For a fixed complex number $c\in\C$ (different from $2 \cos \pi (r a/p \pm sb/q)$ for all pairs $(a,b)$), there is exactly one conjugacy class $[\rho] \in \mathcal{X}^*_{2, (a,b)}$ such that $\tr \rho(\mu) = c$ for each pair $(a,b)$. Then the negative power sum of adjoint Reidemeister torsions with respect to $\mu$ of $[\rho] \in \mathcal{X}^*_2$ with fixed $\tr \rho(\mu) = c$ is equal to 
\begin{eqnarray*}
\pm \sum_{\substack {[\rho] \in \mathcal{X}^*_2 \\ \tr \rho(\mu)=c}}  \frac{1}{(\tau_{\mathrm{ad} \circ \rho})^m} &=&\sum_{\substack{0< a < p \\ a \text{ even}}}  \left( \frac{4\sin^2 \frac{a\pi}{p}} {p} \right)^{m} \sum_{\substack{0< b < q \\ b \text{ even}}} \left( \frac{4\sin^2 \frac{b\pi}{p}} {q} \right)^{m} \\
&& + \, \sum_{\substack{0< a < p \\ a \text{ odd}}} \left( \frac{4\sin^2 \frac{a\pi}{p}} {p}\right)^{m} \sum_{\substack{0< b < q \\ b \text{ odd}}}  \left( \frac{4\sin^2 \frac{b\pi}{q}} {q}\right)^{m}. 
\end{eqnarray*}

Without loss of generality, we can assume that $p$ is odd. Note that $a$ is odd if and only if $p-a$ is even. Since $\sin \frac{a \pi}{p} = \sin \frac{(p-a) \pi}{p}$ we obtain
$$
\sum_{\substack{0< a < p \\ a \text{ odd}}}  (4\sin^2 \frac{a \pi}{p})^m = \sum_{\substack{0< a < p \\ a \text{ even}}} (4\sin^2 \frac{a \pi}{p})^m  = \frac{1}{2}  \sum_{0< a < p} (4\sin^2 \frac{a \pi}{p})^m.
$$
Hence, by Lemma \ref{id1} we have
\begin{eqnarray*}
&& \pm \sum_{\substack {[\rho] \in \mathcal{X}^*_2 \\ \tr \rho(\mu)=c}} \frac{1}{(\tau_{\mathrm{ad} \circ \rho})^m} \\
&=& \frac{1}{2}  \sum_{0< a < p} \left( \frac{4\sin^2 \frac{a\pi}{p}} {p} \right)^{m} \left(  \sum_{\substack{0< b < q \\ b \text{ even}}} \left( \frac{4\sin^2 \frac{b\pi}{p}} {q} \right)^{m} +\sum_{\substack{0< b < q \\ b \text{ odd}}} \left( \frac{4\sin^2 \frac{b\pi}{q}} {q}\right)^{m} \right) \\
&=& \frac{1}{2}  \sum_{0< a < p} \left( \frac{4\sin^2 \frac{a\pi}{p}} {p} \right)^{m}  \sum_{0< b < q} \left( \frac{4\sin^2 \frac{b\pi}{p}} {q} \right)^{m}  \\
&=& \frac{1}{2(pq)^{m-1}} \sum_{l=-\lfloor \frac{m}{p} \rfloor}^{\lfloor \frac{m}{p} \rfloor}  (-1)^{pl} \binom{2m}{m +pl} \sum_{l=-\lfloor \frac{m}{p} \rfloor}^{\lfloor \frac{m}{p} \rfloor}  (-1)^{ql} \binom{2m}{m +ql}
\in\frac{2}{(pq)^{m-1}}\Z.
\end{eqnarray*}
In particular the inverse sum ($m=1$) of adjoint Reidemeister torsions with respect to $\mu$ is equal to
$\pm 2$ for generic $\tr \rho(\mu)$. This is Theorem 7.4 of \cite{TY21-1}. 

\begin{remark}
Let $\lambda = x^p \mu^{-pq}$ be the canonical longitude of $K=K_{p,q}$ corresponding the meridian $\mu = x^{-r} y^s$ where $ps-qr=1$. As in \cite[Section 7]{TY21-1}, given any closed curve $\gamma = \mu^a \lambda^b$ with $a-pqb \not=0$, we can also consider the negative power sum of adjoint Reidemeister torsions with respect to $\gamma$ of $[\rho] \in \mathcal{X}^*_2$ when $\tr \rho(\gamma) = c$ is fixed. Then for $m \ge 0$ and generic $c$, we obtain 
\begin{eqnarray*}
&& \pm \sum_{\substack {[\rho] \in \mathcal{X}^*_2 \\ \tr \rho(\gamma)=c}} \frac{1}{(\tau_{\mathrm{ad} \circ \rho}(K,\gamma))^m} \\
&=& \frac{1}{|a-pqb |^{m-1}}\frac{1}{2(pq)^{m-1}} \sum_{l=-\lfloor \frac{m}{p} \rfloor}^{\lfloor \frac{m}{p} \rfloor}  (-1)^{pl} \binom{2m}{m +pl} \sum_{l=-\lfloor \frac{m}{p} \rfloor}^{\lfloor \frac{m}{p} \rfloor}  (-1)^{ql} \binom{2m}{m +ql}.
\end{eqnarray*}
\end{remark}

\subsubsection{Positive power sums} We have
\begin{eqnarray*}
\pm \sum_{\substack {[\rho] \in \mathcal{X}^*_2 \\ \tr \rho(\mu)=c}}  (\tau_{\mathrm{ad} \circ \rho})^m  &=& \sum_{\substack{0< a < p \\ a \text{ even}}}   \left( \frac{p}{4\sin^2 \frac{a\pi}{p}} \right)^m \sum_{\substack{0< b < q \\ b \text{ even}}}  \left( \frac{q}{4\sin^2 \frac{b\pi}{p}} \right)^m \\
&& + \, \sum_{\substack{0< a < p \\ a \text{ odd}}}   \left( \frac{p}{4\sin^2 \frac{a\pi}{p}} \right)^m \sum_{\substack{0< b < q \\ b \text{ even}}}   \left( \frac{q}{4\sin^2 \frac{b\pi}{p}} \right)^m. 
\end{eqnarray*}
Without loss of generality, we can assume that $p$ is odd. Then 
$$
\sum_{\substack{0< a < p \\ a \text{ odd}}}  \left( \frac{p}{4\sin^2 \frac{a\pi}{p}} \right)^m = \sum_{\substack{0< a < p \\ a \text{ even}}} \left( \frac{p}{4\sin^2 \frac{a\pi}{p}} \right)^m = \text{ rank } V_p(\Sigma_{m+1}).
$$
Hence
\begin{eqnarray*}
&& \pm \sum_{\substack {[\rho] \in \mathcal{X}^*_2 \\ \tr \rho(\mu)=c}}  (\tau_{\mathrm{ad} \circ \rho})^m \\
&=& \text{ rank } V_p(\Sigma_{m+1}) \left( \sum_{\substack{0< b < q \\ b \text{ even}}}  \left( \frac{q}{4\sin^2 \frac{b\pi}{p}} \right)^m + \sum_{\substack{0< b < q \\ b \text{ odd}}} \left( \frac{q}{4\sin^2 \frac{b\pi}{p}} \right)^m \right) \\
&=& \text{rank } V_p(\Sigma_{m+1}) \sum_{0 < b < q} \left( \frac{q}{4\sin^2 \frac{b\pi}{q}} \right)^m \\
&=& 2^{-m} \text{rank } V_p(\Sigma_{m+1})   \text{ rank } V_{2q}(\Sigma_{m+1})\in\frac{1}{2^m}\Z.
\end{eqnarray*}
 
In particular, we obtain 
$\sum_{\substack {[\rho] \in \mathcal{X}^*_2 \\ \tr \rho(\mu)=c}}  (2\tau_{\mathrm{ad} \circ \rho})^m\in\Z$, which answers Conjecture 3.1 in \cite{GKY21-1} for the case of torus knots. 

Note that if $q$ is also odd, then $\text{rank } V_{2q}(\Sigma_{m+1}) = 2^{m+1} \text{rank } V_q(\Sigma_{m+1})$ and so 
$$
\pm \sum_{\substack {[\rho] \in \mathcal{X}^*_2 \\ \tr \rho(\mu)=c}}  (\tau_{\mathrm{ad} \circ \rho})^m = 2 \text{ rank } V_p(\Sigma_{m+1}) \text{ rank } V_q(\Sigma_{m+1})\in 2\Z.
$$

\subsection*{Acknowledgments}
The authors would like to thank the anonymous referee for useful suggestions. 
The first author has been supported by 
JSPS KAKENHI Grant Number JP25K07012. The second author has been supported by a grant from the Simons Foundation (\#708778).


\end{document}